\documentclass[12pt]{article}
\usepackage{srcltx}
\usepackage{times}
\usepackage{amsmath}
\usepackage{latexsym}
\usepackage{amsfonts}
\usepackage{amssymb}
\usepackage{color}

\usepackage{latexsym}
\newcounter{theorem}

\renewcommand{\thetheorem}{\arabic{section}.\arabic{theorem}}

\newenvironment{thm}[1]{\par\addvspace{0.5cm}
    \begin{sloppypar}\refstepcounter{theorem}%
    {\bf #1 \thetheorem.}\it{}}{\end{sloppypar}}

\newcommand{\eh}{\hfill}\newlength{\sperr}

\newenvironment{theorem}{\begin{thm}{Theorem}} {\end{thm}}
\newenvironment{lemma}{\begin{thm}{Lemma}} {\end{thm}}
\newenvironment{corollary}{\begin{thm}{Corollary}} {\end{thm}}

\newenvironment{defi}[1]{\par\addvspace{0.5cm}
\begin{sloppypar}\refstepcounter{theorem}%
{\bf #1 \thetheorem.}\rm{}}{\end{sloppypar}}

\newenvironment{definition}{\begin{defi}{Definition}}{\end{defi}}

\newenvironment{remark}{\begin{defi}{Remark}}{\end{defi}}

\newenvironment{proof}{{\settowidth{\sperr}{\rm Proof}
\par\addvspace{0.3cm}\parbox[t]{1.3\sperr}{\rm P\eh r\eh o\eh o\eh f\eh. }%
}}{\nopagebreak\mbox{}\hfill $\Box$\par\addvspace{0.25cm}}

\newcommand{\al}{\alpha}
\newcommand{\bt}{\beta}

\newcommand{\om}{\omega}
\newcommand{\Om}{\Omega}
\newcommand{\lb}{\lambda}

\newcommand{\vi}{\varphi}

\newcommand{\intl}{\int\limits}

\newcommand{\rone}{\mathbb{R}^1}
\newcommand{\rn}{\mathbb{R}^n}

\newcommand{\vs}{\vspace}

\newcommand{\tcr}{\textcolor{red}}

\newcommand{\ti}{\textit}

\begin{document}
\par\indent

\begin{center}
\LARGE On a Muckenhoupt-type condition for Morrey spaces
\end{center}

\

\centerline{\large  Natasha Samko}

\centerline{\it University of Algarve, Portugal.}

\centerline{\it E-mail: nsamko@ualg.pt, \ nsamko@gmail.com}

\
\noindent{\bf Abstract}

\vspace{2mm}

{\small As is known, the class of weights  for Morrey type spaces $\mathcal{L}^{p,\lb}(\rn) $ for which the maximal and/or singular operators are bounded, is different from the known Muckenhoupt class $A_p$ of such weights for the Lebesgue spaces $L^p(\Om)$. For instance, in the case of power weights $|x-a|^\nu, \ a\in \mathbb{R}^1,$ the singular operator (Hilbert transform) is bounded in  $L^p(\mathbb{R})$, if and only if $-1<\nu <p-1$, while it is bounded in the Morrey space $\mathcal{L}^{p,\lb}(\mathbb{R}), 0\le \lb<1$, if and only if the exponent $\al$ runs the shifted interval  $\lb-1<\nu <\lb+p-1.$ A description of all the admissible weights similar to the Muckenhoupt class $A_p$ is an open problem. In this paper, for the one-dimensional case, we introduce the class $A_{p,\lb}$ of weights, which turns into the Muckenhoupt class $A_p$ when $\lb=0$
and show that the belongness of a weight to $A_{p,\lb}$ is necessary for the boundedness of the Hilbert transform in the one-dimensional case. In the case $n>1$ we also provide some $\lb$-dependent \textit{\`a priori}  assumptions on weights and give some estimates of weighted  norms $\|\chi_B\|_{p,\lb;w}$ of the characteristic functions of balls.

\noindent


\vspace{4mm}

\par\indent
\section{Introduction}
\setcounter{equation}{0} \setcounter{theorem}{0}

The well
known Morrey spaces $\mathcal{L}^{p,\lb}$ introduced in \cite{405a} in
relation to the study of partial differential equations, and presented in
various books, see \cite{187a}, \cite{348a}, \cite{655c}, as well as their various generalizations, were widely
studied during last decades, including the study of classical operators
of harmonic analysis - maximal, singular and potential operators - in these
spaces; we refer for instance to   papers \cite{9e,9f,18f,26c}, \cite{69ad,
69ac,87b,160bz,107f}, \cite{412zz,412z, 412zb,463b,472a,473,504za},
\cite{621a,638a,641a}, where Morrey spaces on metric measure spaces may be
also found. Surprisingly, weighted estimates of these classical operators, in
fact, almost were not studied. Recently, in \cite{539gb} there were proved weighted
$p\to p$-estimates in Morrey spaces for Hardy operators on $\mathbb{R}_+$
and one-dimensional singular
 operators (on $\mathbb{R}$ or on Carleson curves in the complex plane). More general weighted estimates may be found in
\cite{539gc} and \cite{479zzab}. In some papers there were considered special weighted situations when the weight was a power of the function $\vi$ defining the generalized Morrey space, as for instance, in \cite{326bh}, \cite{600zb}.

Let $\Om\subseteq \rn$ and $L^{p,\lb}(\Om ,w)$ denote the classical  Morrey space with weight:
$$L^{p,\lb}(\Om ,w):=\left\{f: \ \|f\|_{p,\lb;w}<\infty\right\}, \ \ \ \ 1\le p<\infty, \ \ 0\le \lb\le 1,$$
by the norm
$$\|f\|_{p,\lb;w}: = \sup\limits_{x,r}\left(\frac{1}{|B(x,r)|^\lb}\intl_{B(x,r)}|f(y)|^p w(y)\,dy\right)^\frac{1}{p},$$
where we always assume that a function $f$ is continues beyond $\Om$ as zero whenever necessary.
We  write $\|f\|_{p,\lb}$ in the non-weighted case $w\equiv 1.$

\vs{4mm}
As shown in \cite{539gb}, the Muckenhoupt class $A_p$ may not be an appropriate class of weights for the case of Morrey spaces.
The appropriate "Muckenhoupt-type" class for the Morrey spaces must  depend on the parameter $\lb.$
As proved in  \cite{539gb} for the one-dimensional case,  the singular integral operator
$$Sf(x)=\frac{1}{\pi}\intl_{\mathbb{R}}\frac{f(t)\,dt}{t-x}$$
is bounded in the space
$L^{p,\lb}(\rone,w)$ with the power weight $w(x)=|x-a|^\nu, \ a\in\Om ,$ if and only if
\begin{equation}\label{1}
\lb-1 <\nu<\lb+p-1
\end{equation}
which is a shifted interval in comparison with the Muckenhoupt condition $-1 <\nu<p-1.$
Thus, condition \eqref{1} partially deletes Muckenhoupt power weights, but on the other hand, adds new ones.

 As is known, a description of all the admissible weights for Morrey spaces, similar to the Muckenhoupt class $A_p$ is an open problem.
 Since the $A_p$-condition
 \begin{equation}\label{next}
A_p: \hspace{15mm}   \sup\limits_{B}\left(\frac{1}{|B|}\intl_{B} w(x)\,dx\right)\left(\frac{1}{|B|}\intl_B w^{1-p^\prime}(y)\,dy\right)^{p-1}  <\infty
\end{equation}
 has the form
 \begin{equation}\label{Muckenhouptrewritten}
 \sup\limits_{B}\frac{1}{|B|}\|v\|_{L^p(B)}\left\|\frac{1}{v}\right\|_{L^{p^\prime}(B)}<\infty, \ \ \ \ \ v=w^\frac{1}{p},
 \end{equation}
 where $\sup$ is taken with respect to all balls in $\rn$, one can expect that the corresponding Muckenhoupt-type  class $A_{p,\lb}$ may be defined by the condition
\begin{equation}\label{Muckenhouptindual}
 A_{p,\lb}: \hspace{15mm} \sup\limits_{B}\frac{1}{|B|}\|v\|_{L^{p,\lb}(B)}\left\|\frac{1}{v}\right\|_{[L^{p,\lb}]^\prime (B)}<\infty, \ \ \ \ \ v=w^\frac{1}{p},
 \end{equation}
 where $[L^{p,\lb}]^\prime$ may stand for the dual (or predual ?) of the Morrey space.  The  preduals of Morrey spaces were studied in   \cite{9eaz},  \cite{9f},
\cite{248b} and  \cite{731c}. Their characterizations are known to be given in capacitory terms and/or in terms of  the so called
 $(q,\lb)$-atomic decompositions, which makes them uneasy in concrete applications.

  We introduce a certain  class $\mathcal{A}_{p,\lb}=\mathcal{A}_{p,\lb}(\rn)$ of weights, which might be conditionally called called a \ti{pre-Muckenhoupt class for Morrey spaces. It turns into the Muckenhoupt class $A_p$ when $\lb=0$
and  we show that the belongness of a weight to this class is necessary  for the one-dimensional singular integral operator (Hilbert transform) to be bounded in the Morrey space}. In the case $n>1$ we also provide some $\lb$-dependent \textit{\`a priori}  assumptions on weights and give some estimates of weighted  norms $\|\chi_B\|_{p,\lb;w}$ of the characteristic functions of balls.

Problems in proving the sufficiency of the introduced $\mathcal{A}_{p,\lb}$-condition are caused by difficulties of transferring  various known properties of the class $A_p$, such as for instance its openness with respect to $p$,  to the class $\mathcal{A}_{p,\lb}.$ We hope to have advances in this relation in another publication.

\section{Some  \textit{\`a priori} assumptions and the class $\mathcal{A}_{p,\lb}$}\label{apriori}
\setcounter{equation}{0} \setcounter{theorem}{0}

The definition \eqref{next}
of the Muckenhoupt class $A_p$ for the spaces $L^p$ (the case $\lb=0$) preassumes that
\begin{equation}
\label{oldass}
  \textrm{the functions} \ \ \  w \ \ \ \textrm{and} \ \ \  w^{1-p^\prime} \ \ \ \textrm{are locally in} \ \ \  L^p.
\end{equation}
What should be
 similar \textit{\`a priori} assumptions for the Morrey spaces?  In the case of the power weight $w=|x-a|^\nu$
  the conditions in \eqref{oldass} mean that $\nu>-n$ and $\nu<n(p-1),$ respectively.
In the case of Morrey spaces, the corresponding interval $(-n,n(p-1))$  should be   shifted to $({n\lb}-n,{n\lb}+n_(p-1))$, as noted
in \eqref{1} in the one-dimensional case $n=1$.  Thus for general weights we expect that the \textit{\`a priori} assumption  $w\in L^1_{\textrm{loc}}$ must be replaced by some more restrictive condition, while the condition
$w^{1-p^\prime}\in L^1_{\textrm{loc}}$ is expected to be weakened, both in dependence on the parameter $\lb.$

As a substitution of the first assumption in \eqref{oldass} we will use now the following natural condition on the weight $w$:
\begin{equation}\label{1intro}
\chi_B\in L^{p,\lb}(\Om ,w) \ \ \ \ \Longleftrightarrow \ \ \ \ \chi_B w^\frac{1}{p}\in L^{p,\lb}(\Om)
\end{equation}
for all the balls $B,$ where  $\chi_E$ denotes the characteristic function of an open set $E\subset \Om.$
As a substitution of another  condition $w^{1-p^\prime}\in L^1_{\textrm{loc}}$ we  introduce the condition
\begin{equation}\label{2intro}
\chi_B \in L^{p,\lb}\left(\Om ,w^{-\frac{1-\lb}{\lb+p-1}}\right) \ \ \ \Longleftrightarrow \ \ \ \
\chi_B w^{-\frac{1}{\lb+p-1}} \in L^{p,\lb}\left(\Om ,w\right),
 \end{equation}
which turns into $w^{1-p^\prime}\in L^1_{\textrm{loc}}$ when $\lb=0$.
With the notation $w(E): = \intl_E w(x)\,dx,$ the conditions  \eqref{1intro} and  \eqref{2intro} have the form
\begin{equation}\label{haveform}
\sup\limits_{B}\frac{w(B\cap B(x_0,r))}{|B|^\lb}<\infty, \ \  \ \ \sup\limits_{B}\frac{w^{-\frac{1-\lb}{\lb+p-1}}(B\cap B(x_0,r))}{|B|^\lb}<\infty, \end{equation}
respectively, where the $\sup$ is taken with respect to all  balls $B\subset \Om .$
\begin{definition}\label{def1}
A weight function $w$ is called \textit{$(p,\lb)$-admissible weight}, if it satisfies the assumptions
\eqref{1intro}-\eqref{2intro}.
\end{definition}

The condition \eqref{1intro} of belongness of functions $\chi_B$ to the weighted space $L^{p,\lb}$ is quite natural. As for the exponent  $-\frac{1-\lb}{\lb+p-1}$ in the condition  \eqref{2intro}, its choice originated in particular from the upper bound in the conditions \eqref{1} known to be necessary and sufficient for power weights.

Now we introduce the class $\mathcal{A}_{p,\lb}$  by the following definition.

\begin{definition}\label{def}
By $\mathcal{A}_{p,\lb}$ we denote the class of $(p,\lb)$-admissible weights satisfying the condition
\begin{equation}\label{Muckenhoupt}
\mathcal{A}_{p,\lb}: \hspace{15mm} \sup\limits_{B}\  \frac{\|\chi_B\|_{p,\lb;w}}{\|\chi_B \|_{p,\lb;w_\ast}}\left(\frac{1}{|B|}\intl_B w^{-\frac{1}{\lb+p-1}}dy\right) <\infty, \ \ \ \  w_\ast=w^{-\frac{1-\lb}{\lb+p-1}},
\end{equation}
where $\sup$ is taken with respect to all balls.
Obviously we obtain the Muckenhoupt class $A_p$ when $\lb=0.$

\begin{remark}\label{rem}
While  $\mathcal{A}_{p,\lb}=A_{p,\lb}=A_p$ in the case $\lb=0$, a comparison of the classes  $\mathcal{A}_{p,\lb}$ and $A_{p,\lb}$ in the case $\lb>0$ is an open question. Note that
\begin{equation}\label{comparison}
A_{p,\lb} \subseteq \mathcal{A}_{p,\lb} \ \ \ \Longleftrightarrow\ \ \ \ \ \intl_B v^{-\frac{p}{\lb+p-1}}\, dy \le C
 \left\|v^{-\frac{1-\lb}{\lb+p-1}}\right\|_{L^{p,\lb}(B)} \left\|\frac{1}{v}\right\|_{[L^{p,\lb}]^\prime(B)}
\end{equation}
the latter inequality may be also rewritten in the form
\begin{equation}\label{comparisonequiv}
\intl_B u\,dy \le C
 \left\|u^{\frac{1-\lb}{p}}\right\|_{L^{p,\lb}(B)} \left\|u^{1-\frac{1-\lb}{p}}\right\|_{[L^{p,\lb}]^\prime(B)}, \ \ \ \ \
 u=v^{-\frac{p}{\lb+p-1}};
\end{equation}
(a \ti{H\"older-type looking inequality}).
\end{remark}

\section{Necessity of the $\mathcal{A}_{p,\lb}$-condition for the Hilbert transform}
\setcounter{equation}{0} \setcounter{theorem}{0}
We pass to the one-dimensional case and consider the singular operator $S$ (Hilbert transform).
 We follow the known approach to prove  the necessity of the $A_p$-conditions for the boundedness of the singular operator known for the Lebesgue spaces, as presented for instance in \cite{317zda}.

We find it convenient to use the notation
$$I=I(x,r)=\{y: x-r<y<x+r\}$$
for the one-dimensional balls.
We assume that the weight $w$ is $(p,\lb)$-admissible in the sense of Definition \ref{def1} which
now means that
\begin{equation}\label{1introcop}
\chi_I\in L^{p,\lb}(\mathbb{R},w)
\end{equation}
and
\begin{equation}\label{2introcop}
\chi_I \in L^{p,\lb}\left(\mathbb{R},w^{-\frac{1-\lb}{\lb+p-1}}\right),
 \end{equation}
for all  intervals $I\subset \mathbb{R}.$

In the sequel, by $\ I^{\prime}$ and $\ I^{\prime\prime}$ we denote two arbitrary  adjoint intervals
$$I^\prime=I(x^\prime,r^\prime),\ \  \ I^{\prime\prime}=I(x^{\prime\prime},r^{\prime\prime}) $$
that is, we suppose that either
$x^{\prime}+r^{\prime}=x^{\prime\prime}-r^{\prime\prime}$ or $x^{\prime}-r^{\prime}=x^{\prime\prime}+r^{\prime\prime}.$
We always have the pointwise estimate
\begin{equation}\label{9below}
(S\chi_{I^\prime})(x) \ge \frac{1}{2}\ \ \ \ \textrm{for all}\ \ \ x\in I^{\prime\prime}
\end{equation}
and all $I^\prime$ and $I^{\prime\prime}$ with
$|I^\prime|=|I^{\prime\prime}|\le 1.$

\vs{3mm} Suppose that the singular operator $S$  is bounded in the  weighted Morrey space:
\begin{equation}\label{9}
\|Sf\|_{p,\lb;w}\le k
\|f\|_{p,\lb;w}.
\end{equation}
Note that $k\ge 1$, which follows from the fact that $S^2=-I.$

\begin{remark}\label{rem1}
In the case of Lebesgue spaces ($\lb=0$) it is known that the boundedness \eqref{9} implies both the conditions
\eqref{1introcop} and \eqref{2introcop}. A similar direct proof of the validity of \eqref{1introcop}, for instance, from \eqref{9} does not hold, because it is based on the use of  duality arguments, which fails in the case of Morrey spaces.
 Instead we suppose that \eqref{1introcop} and \eqref{2introcop} \ti{\`a priori} hold.
\end{remark}

\begin{lemma}\label{lem3} Assume that \eqref{9}
holds  and the weight $w$ has the property \eqref{1introcop}. Then
\begin{equation}\label{9n}
\frac{1}{2k}\|\chi_{I^{\prime}}\|_{p,\lb;w}\le \|\chi_{I^{\prime\prime}}\|_{p,\lb;w}\le 2k \|\chi_{I^{\prime}}\|_{p,\lb;w}
\end{equation}
for all adjoint intervals $I^\prime$ and $I^{\prime\prime}$ with equal lengths
$|I^{\prime}|=|I^{\prime\prime}|\le 1.$
\end{lemma}

\begin{proof}
We substitute the function $f=\chi_{I^\prime}$ into \eqref{9}, which is possible by  \eqref{1introcop}, and obtain
$$\sup\limits_{x,r}\frac{1}{r^{\lb}}\intl_{I(x,r)}|S\chi_{I^\prime}(y)|^pw(y)\,dy\le k^p
\sup\limits_{x,r}\frac{1}{r^\lb}\intl_{I(x,r)}|\chi_{I^\prime}(y)| w(y)\,dy.$$
Then moreover
$$\sup\limits_{x,r}
\frac{1}{r^\lb}
\intl_{I(x,r)\cap I^{\prime\prime}}
|S\chi_{I^\prime}(y)|^p w(y)\,dy\le k^p
\sup\limits_{x,r}
\frac{1}{r^\lb}
\intl_{I(x,r)}|\chi_{I^\prime}(y)| w(y)\,dy$$
and consequently
$$\sup\limits_{x,r}\frac{1}{r^\lb}
\intl_{I(x,r)\cap I^{\prime\prime}}w(y)\,dy\le (2k)^p
\sup\limits_{x,r}\frac{1}{r^\lb}\intl_{I(x,r)}|\chi_{I^\prime}(y)| w(y)\,dy$$
by \eqref{9below}, i.e. we arrive at the right-hand side inequality in \eqref{9n}.
Similarly the left-hand side inequality
is proved.
\end{proof}

\begin{theorem}\label{lem4lb}
Let the assumption \eqref{9} hold  with a $(p,\lb)$-admissible weight $w$. Then
\begin{equation}\label{10}
\sup\limits_{I: |I|\le 1}\  \frac{\|\chi_I\|_{p,\lb;w}}{\|\chi_I \|_{p,\lb;w_\ast}}\left(\frac{1}{|I|}\intl_I w^{-\frac{1}{\lb+p-1}}dy\right) \le 2 k<\infty, \ \ \ \  w_\ast=w^{-\frac{1-\lb}{\lb+p-1}}
\end{equation}
with $  k=\|S\|_{L^{p,\lb}(\mathbb{R},w)\to L^{p,\lb}(\mathbb{R},w)}.$
\end{theorem}

\begin{proof}
Let $I^\prime$ and $I^{\prime\prime}$ be two adjoint intervals with $|I^{\prime}|=|I^{\prime\prime}|\le 1$.
Now we substitute $f=\chi_{I^\prime}w^{-\bt}$, where $\bt=\frac{1}{\lb+p-1}$ into \eqref{9}, which is possible by the assumption \eqref{2intro}:
$$\sup\limits_{B}\frac{1}{|B|^\lb}\intl_B \left|S\left(\chi_{I^\prime}w^{-\bt}\right)(y)\right|^pw(y)\,dy
\le k^p\sup\limits_B \frac{1}{|B|^\lb}\intl_B \left|\chi_{I^\prime}(y)w^{-\bt}(y)\right|^pw(y)\,dy
$$
$$=k^p\sup\limits_B \frac{1}{|B|^\lb}\intl_B \chi_{I^\prime}(y)w^{1-\bt p}(y)\,dy$$
where $B=(x-r,x+r)$ is an arbitrary interval. Hence, moreover
$$\sup\limits_{B}\frac{1}{|B|^\lb}\intl_{B\cap I^{\prime\prime}} \left|S\left(\chi_{I^\prime}w^{-\bt}\right)(y)\right|^pw(y)\,dy
\le k^p\sup\limits_B \frac{1}{|B|^\lb}\intl_B \chi_{I^\prime}(y)w^{1-\bt p}(y)\,dy.$$
Similarly to \eqref{9below} we have
$$S\left(\chi_{I^\prime}w^{-\bt}\right)(y)\ge \frac{1}{2|I^\prime|}\intl_{I^\prime}w^{-\bt}(t)\,dt \ \ \ \
\textrm{for} \ \ \ \ y\in I^{\prime\prime}. $$
Consequently,
$$\frac{1}{|I^\prime|^p}\left(\intl_{I^\prime}w^{-\bt}(t)\,dt\right)^p\sup\limits_{B}\frac{1}{|B|^\lb}\intl_{B\cap I^{\prime\prime}}w(y) \,dy
\le (2k)^p\sup\limits_B \frac{1}{|B|^\lb}\intl_B \chi_{I^\prime}(y)w^{\tcr{1-\bt p}}(y)\,dy,$$
i.e.
$$\frac{1}{|I^\prime|}\intl_{I^\prime}w^{-\bt}(t)\,dt
\|\chi_{I^{\prime\prime}}\|_{p,\lb;w}
\le 2k \|\chi_{I^\prime}\|_{p,\lb;w^{1-\bt p}}.$$
Since $\|\chi_{I^{\prime\prime}}\|_{p,\lb;w}\sim \|\chi_{I^{\prime}}\|_{p,\lb;w}$ by f Lemma \ref{lem3},  we arrive at the condition \eqref{10} with $I^\prime$ redenoted by $I$.
\end{proof}

\begin{corollary}\label{cor1}
Let $w$ be a $(p,\lb)$-admissible weight. The condition $w\in \mathcal{A}_{p,\lb}$ is necessary for the boundedness of the singular operator in the weighted Morrey space $L^{p,\lb}(\mathbb{R},w).$
\end{corollary}
\end{definition}

\section{Norms of characteristic functions of balls in Morrey spaces}.
\setcounter{equation}{0} \setcounter{theorem}{0}

In relation to the weighted Morrey-norms of functions $\chi_{B(x,r)}$ appearing in \eqref{Muckenhoupt}, in this section we give some details on estimation of such norms. 

Note that every simple function belongs to non-weighted Morrey spaces, while it is not the case in general for weighted Morrey spaces.
 Any such belongness for functions $\chi_B$   imposes conditions on the weight, which were already discussed in Section \ref{apriori}. The aim of this section is to shed more light on such belongness and to give some estimations of the norms $\|\chi_B\|_{p,\lb;w}$ involved in the $\mathcal{A}_{p,\lb}$-condition  \eqref{Muckenhoupt} in the case $\Om=\rn.$

\subsection{The non-weighted case}
Let $B(x,r):=\{y\in \rn: |y-x|<r\}$. Fix a ball $B(x_0,r_0).$ The following lemma holds.
\begin{lemma}\label{lem1}
Let $1\le p<\infty, \ \ 0\le \lb\le 1$. The formula
\begin{equation}\label{2}
\left\|\chi_{B(x_0,r_0)}\right\|_{p,\lb}=|B(x_0,r_0)|^\frac{n(1-\lb)}{p}=\left(\om_n r_0^n\right)^\frac{1-\lb}{p}.
\end{equation}
is valid, where $\om_n=|\mathbb{S}^{n-1}|.$
\end{lemma}
\begin{proof}
By the definition of the norm we have
\begin{equation}\label{3}
\left\|\chi_{B(x_0,r_0)}\right\|_{p,\lb}=\sup\limits_{x,r}\left(\frac{1}{(\om_nr^n)^\lb}|B(x,r)\cap B(x_0,r_0)|\right)^\frac{1}{p}=\max\{\mathcal{A},\mathcal{B}\},
\end{equation}
where
$$\mathcal{A}=\sup\limits_{0<r<r_0}\left(\frac{1}{(\om_n r^n)^\lb}|B(x,r)\cap B(x_0,r_0)|\right)^\frac{1}{p}, \ \ \
\mathcal{B}=\sup\limits_{r>r_0}\left(\frac{1}{(\om_n r^n)^\lb}|B(x,r)\cap B(x_0,r_0)|\right)^\frac{1}{p}.$$
We have
$$\mathcal{A}\le \sup\limits_{0<r<r_0}\left(\frac{|B(x,r)|}{(\om_n r^n)^\lb}\right)^\frac{1}{p}=\sup\limits_{0<r<r_0}
\left(\om_n r^n\right)^\frac{1-\lb}{p}=\left(\om_n r_0^n\right)^\frac{1-\lb}{p}
.$$ The same estimate
$\mathcal{B}\le \left(\frac{|B(x_0,r_0)|}{\om_nr_0^{n\lb}}\right)^\frac{1}{p}$ for $\mathcal{B}$ is obvious, so that
\begin{equation}\label{4}
\left\|\chi_{B(x_0,r_0)}\right\|_{p,\lb}\le |B(x_0,r_0)|^\frac{n(1-\lb)}{p}=\left(\om_n r_0^n\right)^\frac{1-\lb}{p}.
\end{equation}
To obtain  the  inequality inverse to \eqref{4} we observe that
$$\left\|\chi_{B(x_0,r_0)}\right\|_{p,\lb}\ge \sup\limits_{r}\left(\frac{1}{(\om_n r^n)^\lb}|B(x_0,r)\cap
B(x_0,r_0)|\right)^\frac{1}{p}= \sup\limits_{r}\left(\frac{1}{(\om_n r^n)^\lb}|B(x_0,\min(r,r_0))|\right)^\frac{1}{p}$$
$$=\max\left\{\sup\limits_{0<r<r_0}\left(\frac{1}{(\om_n r^n)^\lb}|B(x_0,r)|\right)^\frac{1}{p},
\sup\limits_{r>r_0}\left(\frac{1}{(\om_n r^n)^\lb}|B(x_0,r_0)|\right)^\frac{1}{p}\right\}.$$
Hence
\begin{equation}\label{2copy}
\left\|\chi_{B(x_0,r_0)}\right\|_{p,\lb}\ge\left(\frac{|B(x_0,r_0)|}{(\om_nr_0^n)\lb}\right)^\frac{1}{p}.
\end{equation}
 Combining this with \eqref{4},
we arrive at \eqref{2}.

\end{proof}

\subsection{The weighted case}
In the weighted case we cannot already write a precise formula of type \eqref{2} localized to the point $x_0$,
since the  values of the weight $w$ at the points $x$ different from $x_0$ may already  heavily influence on the
value of the norm $\left\|\chi_{B(x_0,r_0)}\right\|_{p,\lb;w}.$

With the usual notation $w(E)=\int_E w(x)\,dx$ we can  write
\begin{equation}\label{5}
\left\|\chi_{B(x_0,r_0)}\right\|_{p,\lb;w}= \sup\limits_{x\in\rn, r>0:\atop |x-x_0|<r+r_0}\left(\frac{w(B(x,r)\cap
B(x_0,r_0)}{(\om_nr^n)^\lb}\right)^\frac{1}{p},
\end{equation}
where we took into account that $w(B(x,r)\cap
B(x_0,r_0)=\emptyset$ when $|x-x_0|>r+r_0$;
however, \eqref{5} is just a direct usage of the definition of the norm. From \eqref{5} we can derive the following statement.
\begin{lemma}\label{lem2}
The norm $\left\|\chi_{B(x_0,r_0)}\right\|_{p,\lb;w}$ admits the estimate
\begin{equation}\label{6}
\frac{1}{\om_n^\lb}\sup\limits_{0<r<r_0}\left(\frac{w(B(x_0,r)}{r^{n\lb}}\right)^\frac{1}{p}\le \left\|\chi_{B(x_0,r_0)}\right\|_{p,\lb;w}\le \frac{1}{\om_n^\lb} \sup\limits_{|x-x_0|<2r_0 \atop 0<r<r_0}\left(\frac{w(B(x,r)}{r^{n\lb}}\right)^\frac{1}{p}.
\end{equation}
\end{lemma}
\begin{proof}
The proof is similar to that of Lemma \ref{lem1}: the left-hand side inequality is proved exactly in the same way as the lower bound \eqref{2copy} in Lemma \ref{lem1}, while the validity of the  right-hand side one becomes obvious from the line
$$\left\|\chi_{B(x_0,r_0)}\right\|_{p,\lb;w}\le \max\{\mathcal{A},\mathcal{B}\}$$
with
$$\mathcal{A}=\sup\limits_{|x-x_0|<2r_0\atop 0<r<r_0}\left(\frac{w(B(x,r))}{(\om_n r^n)^\lb}\right)^\frac{1}{p}
$$
and
$$
\mathcal{B}=\sup\limits_{ r>r_0}\left(\frac{w(B(x_0,r_0))}{(\om_n r^n)^\lb}\right)^\frac{1}{p}=
\left(\frac{w(B(x_0,r_0))}{(\om_nr_0^n)^\lb}\right)^\frac{1}{p}\le \mathcal{A}.$$
\end{proof}
The following corollary  provides  a sufficient condition on the weight function $w$ for which bounded functions with a compact support belong to the weighted space $L^{p,\lb}(\rn,w).$
\begin{corollary}\label{cor}
For the characteristic function $\chi_{B(x_0,r_0)}$ of a ball $B(x_0,r_0)$ to belong to the space $L^{p,\lb}(\rn,w)$, the condition
\begin{equation}\label{7}
\sup\limits_{|x-x_0|<2r_0 \atop 0<r<r_0}\frac{w(B(x,r))}{r^{n\lb}}<\infty
\end{equation}
is sufficient, and the condition
\begin{equation}\label{8}
\sup\limits_{0<r<r_0}\frac{w(B(x_0,r))}{r^{n\lb}}<\infty
\end{equation}
is necessary.
\end{corollary}

\begin{remark}\label{rem}
Let  $w(x)=|x-a|^\nu g(x),$ where $ \ a\in \rn$, $g$ is a bounded function with compact support, and   $\nu>-n$. When $a\in B(x_0,r_0),$ then
\begin{equation}\label{8hg}
\chi_{B(x_0,r_0)}\in L^{p,\lb}(\rn,w)
\end{equation}
 if and only if
\begin{equation}\label{8}
\nu\ge {n\lb}-n,
\end{equation}
and
\begin{equation}\label{8su}
\|\chi_B\|_{p,\lb;w}\sim |B|^\frac{n+\nu-{n\lb}}{np}
\end{equation}
in this case.
When $|x_0-a|$ is large enough,  $|x_0-a|>2r_0,$ the inclusion \eqref{8hg} holds for any $\nu>-n$.

\end{remark}

\begin{proof}
Direct estimations via the passage to polar coordinates, dilation change of variables and rotation yield
$$\frac{w(B(x,r))}{(\om_n r^n)^\lb}=\frac{1}{r^{n\lb}}\intl_{|y-x|<r}|y-a|^\nu\,dy=\frac{1}{r^{n\lb}}\intl_{|y|<r}|y-(x-a|^\nu\,dy
=\frac{|x-a|^{n+\nu}}{r^{n\lb}}\intl_{|y|<\frac{r}{|x-a|}}|y-e_1|^\nu\,dy,
$$
where $e_1=(1,0,...,0)$. The remaining integral $I(t)=\intl_{|y|<t}|y-e_1|^\nu\,dy$ is estimated by standard means:
$I(t)\sim \left\{\begin{array}{ll}t^n, \ & 0<t<1\\  t^{n+\nu}, \ & t>1\end{array}\right.,$
where $I(t)\sim \textrm{R.H.S.}$ means that $c_1 \textrm{R.H.S.}\le I(t)\le c_2\textrm{R.H.S.}$ with $c_1$and $c_2$ not depending on $t$.
Then
$$\frac{w(B(x,r))}{r^{n\lb}} \sim \left\{\begin{array}{ll}r^{n-{n\lb}}|x-a|^\nu, \ & r\le |x-a|\\  r^{n+\nu-{n\lb}}, \ & r\ge |x-a|\end{array}\right.$$
from where the statement of the remark follows, with the necessity statement checked directly at the point $x=x_0.$
\end{proof}

%\bibliographystyle{plain}

%\bibliography{Morrey_weighted}

\begin{thebibliography}{10}

\bibitem{9eaz}
D.R. Adams.
\newblock A note on {C}hoquet integrals with respect to {H}ausdorff capacity.
\newblock In {\em Function spaces and applications ({L}und, 1986)}.

\bibitem{9e}
D.R. Adams.
\newblock A note on {R}iesz potentials.
\newblock {\em Duke Math. J.}, 42(4):765--778, 1975.

\bibitem{9f}
D.R. Adams and J.~Xiao.
\newblock Nonlinear potential analysis on {M}orrey spaces and their capacities.
\newblock {\em Indiana Univ. Math. J.}, 53(6):1631--1666, 2004.

\bibitem{18f}
J.~Alvarez.
\newblock The distribution function in the {M}orrey space.
\newblock {\em Proc. Amer. Math. Soc.}, 83:693--699, 1981.

\bibitem{26c}
H.~Arai and T.~Mizuhara.
\newblock {M}orrey spaces on spaces of homogeneous type and estimates for
  $\square_b$ and the {C}auchy-{S}zego projection.
\newblock {\em Math. Nachr.}, 185(1):5--20, 1997.

\bibitem{69ad}
V.I. Burenkov, V.~Guliev, and H.~Guliyev.
\newblock Necessary and sufficient conditions for the boundedness of fractional
  maximal operators in local {M}orrey-type spaces.
\newblock {\em J. Comp. Appl. Math.}, 208(1):280 – 301, 2007.

\bibitem{69ac}
V.I. Burenkov and H.~Guliyev.
\newblock Necessary and sufficient conditions for boundedness of the maximal
  operator in local {M}orrey-type spaces.
\newblock {\em Studia Math.}, 163(2):157--176, 2004.

\bibitem{87b}
F.~Chiarenza and M.~Frasca.
\newblock Morrey spaces and {H}ardy-{L}ittlewood maximal function.
\newblock {\em Rend. Math.}, 7:273--279, 1987.

\bibitem{160bz}
G.~Di~Fazio and M.A. Ragusa.
\newblock Commutators and {M}orrey spaces.
\newblock {\em Bollettino U.M.I.}, 7(5-A):323--332, 1991.

\bibitem{107f}
Y.~Ding and S.~Lu.
\newblock Boundedness of homogeneous fractional integrals on ${L}^p$ for
  $n/a<p$.
\newblock {\em Nagoya Math. J.}, 167:17--33, 2002.

\bibitem{187a}
M.~Giaquinta.
\newblock {\em Multiple integrals in the calculus of variations and non-linear
  elliptic systems}.
\newblock Princeton Univ. Press, 1983.

\bibitem{248b}
E.A. Kalita.
\newblock Dual {M}orrey spaces.
\newblock {\em Dokl. Akad. Nauk}, 361(4):447--449, 1998.

\bibitem{317zda}
V.~Kokilashvili and A.~Meskhi.
\newblock A note on the boundedness of the {H}ilbert transform in weighted
  grand {L}ebesgue spaces.
\newblock {\em Georgian Math. J.}, 16(3):547--551, 2009.

\bibitem{326bh}
Y.~Komori and S.~Shirai.
\newblock Weighted {M}orrey spaces and a singular integral operator.
\newblock {\em Math. Nachr.}, 282(2):219--231, 2009.

\bibitem{348a}
A.~Kufner, O.~John, and S.~Fu$\check{c}$ik.
\newblock {\em Function {Spaces}}.
\newblock Noordhoff International Publishing, 1977.
\newblock 454 + XV pages.

\bibitem{405a}
C.B. Morrey.
\newblock On the solutions of quasi-linear elliptic partial differential
  equations.
\newblock {\em Amer. Math. Soc.}, 43:126--166, 1938.

\bibitem{412zz}
E.~Nakai.
\newblock Hardy-{L}ittlewood maximal operator, singular integral operators and
  the {R}iesz potentials on generalized {M}orrey spaces.
\newblock {\em Math. Nachr.}, 166:95--103, 1994.

\bibitem{412z}
E.~Nakai.
\newblock On generalized fractional integrals.
\newblock {\em Taiwanese J. Math.}, 5(3):587--602, 2001.

\bibitem{412zb}
E.~Nakai and H.~Sumitomo.
\newblock On generalized {R}iesz potentials and spaces of some smooth
  functions.
\newblock {\em Sci. Math. Jpn.}, 54(3):463--472, 2001.

\bibitem{463b}
D.K. Palagachev and L.G. Softova.
\newblock Singular integral operators, {M}orrey spaces and fine regularity of
  solutions to {P}{D}{E}'s.
\newblock {\em Potential Analysis}, 20:237--263, 2004.

\bibitem{472a}
J.~Peetre.
\newblock On convolution operators leaving {$\mathcal{L}^{p,\lambda}$} spaces
  invariant.
\newblock {\em Annali di Mat. Pura ed Appl.}, 72(1):295--304, 1966.

\bibitem{473}
J.~Peetre.
\newblock On the theory of {$\mathcal{L}_{p,\lambda}$} spaces.
\newblock {\em Function. Analysis}, 4:71--87, 1969.

\bibitem{479zzab}
E.-L. Persson and N.~Samko.
\newblock Weighted {H}ardy and potential operators in the generalized {M}orrey
  spaces.
\newblock {\em J. Math. Anal. Appl.}, 377:792--806, 2011.

\bibitem{504za}
M.A. Ragusa.
\newblock Commutators of fractional integral operators on {V}anishing-{M}orrey
  spaces.
\newblock {\em J. of Global Optim.}, 40(1-3):361 -- 368, 2008.

\bibitem{539gc}
N.G. Samko.
\newblock Weighted {H}ardy and potential operators in {M}orrey spaces.
\newblock {\em J. Funct. Spaces and Appl.}
\newblock to appear.

\bibitem{539gb}
N.G. Samko.
\newblock Weighted {H}ardy and singular operators in {M}orrey spaces.
\newblock {\em J. Math. Anal. and Appl.}, 350:56--72, 2009.

\bibitem{600zb}
Y.~Sawano and H.~Tanaka.
\newblock Morrey spaces for non-doubling measures.
\newblock {\em Acta Math. Sin. (Engl. Ser.)}, 21(6):1535--1544, 2005.

\bibitem{621a}
S.~Shirai.
\newblock Necessary and sufficient conditions for boundedness of commutators of
  fractional integral operators on classical {M}orrey spaces.
\newblock {\em Hokkaido Math. J.}, 35(3):683--696, 2006.

\bibitem{638a}
S.~Spanne.
\newblock Some function spaces defined by using the mean oscillation over
  cubes.
\newblock {\em Ann. Scuola Norm. Sup. Pisa}, 19:593--608, 1965.

\bibitem{641a}
G.~Stampacchia.
\newblock The spaces ${L}^{p,\lambda}, {N}^{(p, \lambda)}$ and interpolation.
\newblock {\em Ann. Scuola Norm. Super. Pisa}, 3(19):443--462, 1965.

\bibitem{655c}
M.~E. Taylor.
\newblock {\em Tools for {P}{D}{E}: {P}seudodifferential {O}perators,
  {P}aradifferential {O}perators, and {L}ayer {P}otentials}, volume~81 of {\em
  Math. Surveys and Monogr}.
\newblock AMS, Providence, R.I., 2000.

\bibitem{731c}
C.~T. Zorko.
\newblock Morrey space.
\newblock {\em Proc. Amer. Math. Soc.}, 98(4):586--592, 1986.

\end{thebibliography}

\end{document}